\documentclass[a4paper,12pt]{amsart}

\usepackage[latin1]{inputenc}
\usepackage{amsmath,amssymb%}%,
 ,amsthm
 ,amscd
}

\date{2012.12}
\def\Z{\mathbb Z}
\def\L{\mathbb L}

\def\b{{\bf b}}
\def\x{{\bf x}}
\def\y{{\bf y}}
\def\pf{{\;+_{_F}\:}}

\newtheorem{thm}{Theorem}[section]
\newtheorem{lem}{Lemma}[section]

\newtheorem{prop}{Proposition}[section]

\newtheorem{Def}{Definition}[section]
\newtheorem{conj}{Conjecture}[section]

\author[M. Nakagawa, H. Naruse]
{Masaki Nakagawa
\and
Hiroshi Naruse
}

\email{nakagawa@okayama-u.ac.jp}
\email{rdcv1654@okayama-u.ac.jp}

\title
{Universal factorial Schur $P,Q$-functions\\
 and their duals}

\address{
Graduate School of Education, Okayama University,
Tsushima-naka, Kitaku, Okayama 700-8530, JAPAN}

\begin{document}
\maketitle
\begin{abstract}

We define universal factorial Schur $P,Q$-functions 
and their duals,
which specialize to generalized (co)-homology
``Schubert basis" for loop spaces of the classical groups. 
We also investigate some of their properties.

\end{abstract}

\setcounter{section}{-1}

\section{Introduction}

By the work of Pragacz \cite{Pra},
Schur $P,Q$-functions are known as Schubert basis of the cohomology rings of 
maximal orthogonal or Lagrangian grassmannians.
There are also factorial versions and they can be interpreted as
torus equivariant cohomology Schubert basis.
Recently these are extended to equivariant $K$-theory case \cite{Ike-Nar}.
On the way of studying known results \cite{Cla,Bak} on 
$K$-homology of loop  spaces of the classical groups
$Sp=Sp(\infty)$ and $SO=SO(\infty)$, 
the first author noticed that the $K$-homology of these spaces can be
realized as subspaces of the ring of symmetric functions \cite{Nak}. 
Combining these together with the Cauchy type kernel, we recognized that
these realization can be extended to generalized (co)-homology setting.

\section{Lazard ring $\L$ and formal group law}%sec1%%%%%%%%%%%%%%%%%%%%%%%%%%%%%%%%%%%%%%%%%%%%%%%%

In \cite{Laz} Lazard considered a universal commutative formal group law 
of rank one now called the Lazard ring.
Let $\L=\L_{*}$ be the Lazard ring and
$F_{\L}(u,v)$ be the universal formal group law (For a construction 
and basic properties of $\L$, see e.g. \cite{Lev-Mor}):

$$F_{\L}(u,v)=\sum_{i,j} a_{i,j} u^i v^j\in \L[[u,v]].$$

This is a formal power series in $u,v$ with coefficients $a_{i,j}$ of formal variables
which satisfies the axiom of formal group law.
We will write $a\pf b=F_{\L}(a,b)$
and $\overline{a}=\chi_{_{\L}}(a)$, the inverse of $a$, i.e. which satisfies
$a\pf \overline{a}=0$.
It is known that
$\overline{a}\in \L[[a]]$ is a formal power series in $a$ with initial term $-a$ and
first few terms appear in P.41 of \cite{Lev-Mor}.

The grading of $\L_{*}$ is as follows. The homological degree is defined by
$\deg_h(a_{i,j})=i+j-1$.
It is known that $\L_{*}$ is isomorphic to the polynomial ring 
in countably infinite number of variables with  integer coefficients.
% $\Z$.
There is a symmetric function realization of this ring by Lenart \cite{Len}.

\section{Cohomological basis}%sec2%%%%%%%%%%%%%%%%%%%%%%%%%%%%%%%%%%%%%%%%%%%%%%%%

We provide the variables
$\x=(x_1,x_2,\ldots)$ and $\b=(b_1,b_2,\ldots)$ with degree
$\deg(x_i)=\deg(b_i)=1$.

\subsection{Definition of $P^\L_\lambda,Q^\L_\lambda$}

\label{formal factorial powers}

For an integer $k\geq 1$, we set
$[t|\b]^k=\displaystyle\prod_{i=1}^{k}(t\pf b_i)$ and
$[[t|\b]]^{k+1}=(t\pf t)[t|\b]^k$.
We also set $[t|\b]^0=[[t|\b]]^0=1$.
For a partition $\lambda=(\lambda_1,\ldots,\lambda_r)$, 	
we set
$$[x|\b]^\lambda:=\displaystyle\prod_{i=1}^{r} [x_i|\b]^{\lambda_i}
\text{ and }
[[x|\b]]^\lambda:=\displaystyle\prod_{i=1}^{r} [[x_i|\b]]^{\lambda_i}.$$

Let  
$\lambda=(\lambda_1,\ldots,\lambda_r)$
be a strict partition of length $r$, i.e. 
a sequence of positive integers 
such that $\lambda_1>\cdots>\lambda_r>0.$ 
We denote by $\mathcal{SP}$ the set of all strict partitions, and
$\mathcal{SP}_n$ the subset of $\mathcal{SP}$ consisting of 
strict partitions of length $r\leq n.$
The following definition with coefficients in $\L[[\b]]$  
was suggested by Anatol Kirillov, 
and
we thank him for this.

\begin{Def}(Universal factorial Schur $P,Q$-functions)\label{defPQ} 
(cf.\cite{Ike-Nar} Def. 2.1)\\
For a strict partition $\lambda=(\lambda_1,\ldots,\lambda_r)$  ($r\leq n$), we define 
\begin{eqnarray*}
P^\L_\lambda(x_1,\ldots,x_n |\:\b)&:=&
\displaystyle\frac{1}{(n-r)!}
\sum_{w\in S_n}w\left[
[x\:|\:\b]^{\lambda}
\prod_{i=1}^r
\prod_{j=i+1}^n \frac{x_{i}\pf x_{j}}{x_{i}\pf \overline{x}_{j}}\right],
\label{def:GP}\\
%%%%%%%%%%%%%%%%%%%%%%%%%%%%%%%%%%%%%%%%%%%%%%%%%%%%%%%%
Q^\L_\lambda(x_1,\ldots,x_n |\:\b)&:=&
\displaystyle\frac{1}{(n-r)!}
\sum_{w\in S_n}w\left[
[[x\:|\:\b]]^{\lambda}
\prod_{i=1}^r
\prod_{j=i+1}^n \frac{x_{i}\pf x_{j}}{x_{i}\pf \overline{x}_{j}}
\right].\label{def:GQ}\\
%%%%%%%%%%%%%%%%%%%%%%%%%%%%%%%%%%%%%%%%
\end{eqnarray*}
\end{Def}

These are symmetric functions in the variables $(x_1,\ldots,x_n)$ and
formal series in $b_1,b_2,\ldots,b_{\lambda_1}$
by definition.
We call these formal series the 
{\it universal factorial Schur $P,Q$- functions}.
Note that 
$$t\pf \overline{s}=(t-s)(1+\text{higher degree terms in } t\text{ and }s
\text{ with coefficients in }\L).$$ 
This means that
$P^\L_\lambda(x_1,\ldots,x_n|\:\b)$ and
$Q^\L_\lambda(x_1,\ldots,x_n|\:\b)$
are well defined in $\L[[\b]][[x_1,\ldots,x_n]]$.

We also set 

$P^\L_\lambda(x_1,\ldots,x_n):=P^\L_\lambda(x_1,\ldots,x_n|\:0)$
,
$Q^\L_\lambda(x_1,\ldots,x_n):=Q^\L_\lambda(x_1,\ldots,x_n|\:0)$.

\subsection{$\L$-supersymmetric series}

%\section{$\L$-supersymmetric function}

$\L$-supersymmetric formal series is defined as follows.

\begin{Def}($\L$-supersymmetric series)

A formal power series $f(x_1,\ldots,x_n)$ with coefficients in %(an extention of) 
$\L$ 
 is called
{$\L$-supersymmetric} if

\begin{itemize}
\item[(1)] $f(x_1,\ldots,x_n)$ is symmetric in variables $x_1,\ldots,x_n$ , and
\item[(2)] $f(t,\overline{t},x_3,\ldots,x_n)$ does not depend on $t$.
\end{itemize}

\end{Def}

We denote by
$\Gamma^{\L}(\x_n)$ the space of 
$\L$-supersymmetric formal series in $x_1,\ldots,x_n$ with coefficients in $\L$.
We also define
$\Gamma_{+}^{\L}(\x_n)$ to be the subspace of $\Gamma^{\L}(\x_n)$ consisting of 
$f(x_1,\ldots,x_n)\in \Gamma^{\L}(\x_n)$ such that
$f(t,x_2,\ldots,x_n)-f(0,x_2,\ldots,x_n)$ is divisible by $t \pf t$.

% Note that
% $a-\overline{b}$ is divisible by $a\pf b$, because 
% it becomes zero when $a=\overline{b}$.
% Therefore $t-\overline{t}$ is divisible by $t\pf t$.

\vspace{0.5cm}

\begin{prop}
$P^{\L}_\lambda(x_1,\ldots,x_n)\in \Gamma^{\L}(\x_n)$ and 
$Q^{\L}_\lambda(x_1,\ldots,x_n)\in \Gamma_{+}^{\L}(\x_n)$.

\end{prop}
\begin{proof}
%Proof) 
Similar to the proof of Prop 3.1 and Prop. 3.2 in \cite{Ike-Nar}.

\end{proof}
%\vspace{0.5cm}

For a positive integer $n$ we set $\rho_{n}=(n,n-1,\ldots,1)$.
Let $\mathcal{P}_n$ be the set of all partitions of length $\leq n$.

For a partition $\lambda=(\lambda_1,\ldots,\lambda_n)\in \mathcal{P}_n$, we define
$$s^{\L}_\lambda(x_1,\ldots,x_n|\:\b):=
\displaystyle
\sum_{w\in S_n}w\left[\frac{[x\:|\:\b]^{\lambda+\rho_{n-1}}}
{\prod_{1\leq i<j\leq n}(x_{i}\pf \overline{x}_{j})}
\right]
\text{ and }
s^{\L}_\lambda(x_1,\ldots,x_n):=s^{\L}_\lambda(x_1,\ldots,x_n|\:0).
$$

Note that 
$s^{\L}_\lambda(x_1,\ldots,x_n)=s_\lambda(x_1,\ldots,x_n)+$ higher terms and
$s_{\emptyset}^{\L}(x_1,\ldots,x_n)\neq 1$.

\begin{prop}(Factorization) (cf.\cite{Ike-Nar} Prop. 2.3.)

\begin{itemize}

\item[(1)] For a positive integer $n$ we have
$$P^{\L}_{\rho_{n-1}}(x_1,\ldots,x_n)=
\left(\displaystyle
\prod_{1\leq i<j\leq n}(x_i\pf x_j)\right)
s_{\emptyset}^{\L}(x_1,\ldots,x_n),$$
$$Q^{\L}_{\rho_{n}}(x_1,\ldots,x_n)=
\left(\displaystyle
\prod_{1\leq i\leq j\leq n}(x_i\pf x_j)\right)
s_{\emptyset}^{\L}(x_1,\ldots,x_n).$$
\item[(2)] For a partition  $\lambda=(\lambda_1\geq \lambda_2\geq\cdots\geq\lambda_n\geq0)$,
$$P^{\L}_{\rho_{n-1}+\lambda}(x_1,\ldots,x_n|\:\b)=
%P^{\L}_{\rho_{n-1}}(x_1,\ldots,x_n) 
\left(\displaystyle
\prod_{1\leq i<j\leq n}(x_i\pf x_j)\right)
s^{\L}_\lambda(x_1,\ldots,x_n|\:\b),$$
$$Q^{\L}_{\rho_{n}+\lambda}(x_1,\ldots,x_n|\:\b)=
%Q^{\L}_{\rho_{n}}(x_1,\ldots,x_n)
\left(\displaystyle
\prod_{1\leq i\leq j\leq n}(x_i\pf x_j)\right)
 s^{\L}_\lambda(x_1,\ldots,x_n|\:\b).$$
\end{itemize}
\end{prop}

\begin{proof}
%Proof)

These equations follow from the definition.

\end{proof}

%\vspace{1cm}

\begin{thm}(Basis theorem)

\begin{itemize}
\item[(1)] The polynomials $P^{\L}_\lambda(x_1,\ldots,x_n)$ 
($\lambda\in \mathcal{SP}_n$)
form a formal $\L$-basis of $\Gamma^{\L}(\x_n)$.

\item[(2)] The polynomials $Q^{\L}_\lambda(x_1,\ldots,x_n)$ 
($\lambda\in \mathcal{SP}_n$)
 form a formal $\L$-basis of $\Gamma^{\L}_{+}(\x_n)$.
\end{itemize}

\end{thm}

\begin{proof}
We can follow the same strategy as in Theorem 3.1 of \cite{Ike-Nar}.
\end{proof}

\subsection{Vanishing property}

For a strict partition $\mu=(\mu_1,\ldots,\mu_r)$ of length $r$, we set
$\overline{\b}_\mu=(\overline{b}_{\mu_1},\ldots,\overline{b}_{\mu_r},0,0,\ldots)$.
We also set \\
${\rm sh}(\mu)=(\mu_1+1,\ldots,\mu_r+1)$ if $r$ is even and
${\rm sh}(\mu)=(\mu_1+1,\ldots,\mu_r+1,1)$ if $r$ is odd 
(We consider only even variable case for 
$P^\L_\lambda(x_1,\ldots,x_{2n}|\:\b)$ because of stability. \cite{Ike-Nar} Remark 3.1)
.

\begin{prop}(Vanishing)\label{vanishing}
\begin{itemize}
\item[(1)] $P^\L_\lambda(\overline{\b}_{{\rm sh}(\mu)}|\:\b)=0$ 
if $\mu\not \supset \lambda$
and \\
$P^\L_\lambda(\overline{\b}_{{\rm sh}(\lambda)}|\:\b)=\displaystyle
\prod_{i=1}^{r}
\left(
\prod_{1\leq j\leq \lambda_i,j\neq \lambda_p+1,
\:\text{ for } i< p\leq r}(\overline{b}_{\lambda_i+1}\pf b_j)
\prod_{j=i+1}^{r}\left( 
\overline{b}_{\lambda_i+1}\pf \overline{b}_{\lambda_j+1}
\right)
\right)
$.
\item[(2)] $Q^\L_\lambda(\overline{\b}_\mu|\:\b)=0$ if $\mu\not \supset \lambda$
and \\
$Q^\L_\lambda(\overline{\b}_\lambda|\:\b)=\displaystyle
\prod_{i=1}^r\left(
\prod_{1\leq j\leq \lambda_i-1,j\neq \lambda_p,\:\text{ for } i< p\leq r}
\left(\overline{b}_{\lambda_i}\pf b_j\right)
\prod_{j=i}^{r} 
\left(\overline{b}_{\lambda_i}\pf \overline{b}_{\lambda_j}\right)
\right)
$.

\end{itemize}
\end{prop}

\begin{proof}

We will only prove (1). The proof of (2) is similar. We may assume
that the length of $\lambda$ is less than or equal to $n$.
If $\mu\not \supset \lambda$ then there is an index $k\leq r$ such that
$\mu_k<\lambda_k$.
It is easy to see that $[t|\b]^{\lambda_k}$ becomes zero 
when $t=\overline{b}_{\mu_k+1}$.
This means that $P^{\L}_{\lambda}(\overline{\b}_{{\rm sh}(\mu)}|\b)=0$.
For the case of $\mu=\lambda$,
we can see that the terms in $P^{\L}_{\lambda}(x_1,\ldots,x_r|\b)$ other than $w=e$
become zero when we evaluate them at 
$(x_1,\ldots,x_r)=\overline{\b}_{{\rm sh}(\lambda)}$.
The term corresponding to $w=e$ becomes the desired value 
(cf. \cite{Ike-Nar} Prop.7.1).

\end{proof}

%\pagebreak

\section
{Homological basis 
$\widehat{p}^{\L}_\lambda(\y|\:\b),\widehat{q}^{\L}_\lambda(\y|\:\b)$.}
%sec3%%%%%%%%%%%%%%%%%%%%%%%%%%%%%%%%%%%%%%%%%%%%%%%%%

We use countably infinite variables 
$\y=(y_1,y_2,\ldots)$. The homological degree is $\deg_h(y_i)=1$ for $i=1,2,\ldots$.
Let $\Lambda(\y)$ be the ring of symmetric functions in $\y$ 
with coefficients in $\Z$.

\subsection{One row case}

\begin{Def}
$$\Delta(t;\y):=
\prod_{j=1}^{\infty}\frac{1-\overline{t} y_j}{1-t y_j}.$$
\end{Def}

As $\{ [t|\b]^k\}_{k=0,1,2,\ldots}$ is a formal
$\L[[\b]]\hat{\otimes}\Lambda(\y)$-basis of the ring 
$\L[[\b]][[t]] \hat{\otimes} \Lambda(\y)$ (see Lemma below),
we can expand 
$$\Delta(t;\y)
=\sum_{k=0}^\infty [t|\b]^k\;\widehat{q}^\L_{k}(\y|\:\b)
$$
to get $\widehat{q}^\L_{k}(\y|\:\b)\in \L[[\b]] \hat{\otimes} \Lambda(\y)$.
$\widehat{q}^\L_{k}(\y|\:\b)$ is a formal power series with coefficients in 
$\L[[\b]]$ except 
$\widehat{q}^\L_{0}(\y|\:\b)=1$.

\vspace{0.5cm}

\begin{lem}
If $f(x)$ is in $\L[[x]]$ (or its extension) and becomes zero when $x=\overline{t}$ then $f(x)$ is divisible by $x\pf t$.

\end{lem}
\begin{proof}

By the assumption $f(x)$ is divisible by $x-\overline{t}$.
Since  $x-\overline{t}$ is equal to
$(x\pf t)(1+\text{ higher terms})$,
the result follows.

\end{proof}

Using the lemma above we can show  that $\Delta(t;\y)-1$ is divisible by $t\pf t$.
Continuing this kind of argument we can define
$\widehat{p}^\L_{k}(\y|\b)$
by the following identity.

%$\{[[t|\b]]^k\}_{k=0,1,2,\ldots}$ and 

$$\Delta(t;\y)=\sum_{k=0}^\infty [[t|\b]]^k\;\widehat{p}^\L_{k}(\y|\:\b).$$

Define
$$\Gamma_{\L}:=\text{ the }\L[[\b]]\text{-subalgebra of } 
\L[[\b]]\hat{\otimes} \Lambda(\y)
\text{ generated by } \widehat{p}^\L_{k}(\y|\:\b),\; k=0,1,2,\ldots\:,$$
$$\Gamma_{\L}^{+}:=\text{ the }\L[[\b]]\text{-subalgebra of }
 \L[[\b]]\hat{\otimes} \Lambda(\y)
\text{ generated by } \widehat{q}^\L_{k}(\y|\:\b),\; k=0,1,2,\ldots\:.$$

For a partition $\lambda=(\lambda_1,\ldots,\lambda_r)$ of length $r \leq n$, we define

$$\widehat{p}^{[\lambda]}(\y|\:\b):=\prod_{i=1}^{r} \widehat{p}^\L_{\lambda_i}(\y|\:\b)
\;\;\text{ and }\;\;
\widehat{q}^{[\lambda]}(\y|\:\b):=\prod_{i=1}^{r} \widehat{q}^\L_{\lambda_i}(\y|\:\b).$$

We define the subspaces of $\Gamma_{\L}$ and $\Gamma_{\L}^{+}$ by using the set 
$\mathcal{P}_n$ of
partitions of length $\leq n$ as follows.

\begin{center}
$
\begin{array}{lcl}
\Gamma_{\L}^{(n)}(\y)&:=&\displaystyle\sum_{\lambda\in \mathcal{P}_n} 
\L[[\b]]\: \widehat{p}^{[\lambda]}(\y|\:\b)\subset \Gamma_{\L}, \text{ and }\\[0.5cm] 
\Gamma_{\L}^{(n),+}(\y)&:=&\displaystyle\sum_{\lambda\in \mathcal{P}_n} 
\L[[\b]]\: \widehat{q}^{[\lambda]}(\y|\:\b)\subset \Gamma_{\L}^{+}.
\end{array}
$
\end{center}

\subsection{Kernel}

We  consider the iterated product of $\Delta(x_i;\y)$'s, and their limit.

$$\Delta(x_1,\ldots,x_n;\y):=\prod_{i=1}^{n} \Delta(x_i;\y),\;\;
\Delta(\x;\y):=\lim_{\substack{\longleftarrow\\n}} \Delta(x_1,\ldots,x_n;\y).$$

\begin{prop}($\L$-supersymmetricity)
\begin{itemize}
\item[(1)] 
$\Delta(x_1,\ldots,x_n;\y)$ is $\L$-supersymmetric in variables $(x_1,\ldots,x_n)$.

\item[(2)]
$\Delta(x_1,\ldots,x_n;\y)\in \Gamma^\L_{+}(\x_n) \otimes_{\L[[\b]]} \Gamma_{\L}^{(n)}(\y)$
and
$\Delta(x_1,\ldots,x_n;\y)\in \Gamma^\L_{}(\x_n) \otimes_{\L[[\b]]} \Gamma_{\L}^{(n),+}(\y)$.

\end{itemize}

\end{prop}

Proof)

(1) follows from the definition.

(2) $\Delta(x_1,\ldots,x_n;\y)$ has the property that
$\Delta(t,x_2,\ldots,x_n;\y)-\Delta(0,x_2,\ldots,x_n;\y)$ is divisible by 
$t\pf t$ (see Lemma 3.1).

%\pagebreak

% We use a Cauchy type kernel 
%$\Delta=\prod\displaystyle\frac{1-\bar{x}_iy_j}{1-x_i y_j}$ 
% to define universal factorial
% dual Schur $P,Q$-functions
% $\widehat{p}^\L_\lambda(\y|\b)$ and
% $\widehat{q}^\L_\lambda(\y|\b)$
% as follows.

% where $E=(E^*,E_*)$ is a generalized (co-)homology theory
%\cite{Adams}.

\vspace{1cm}

We can define stable limit functions
$$Q^\L_\lambda(\x\:|\:\b):=\lim_{\substack{\longleftarrow\\n}} Q^\L_\lambda(x_1,\ldots, x_n|\:\b)
\;\text{
 and }
P^\L_\lambda(\x\:|\:\b):=\lim_{\substack{\longleftarrow\\n}} 
P^\L_\lambda(x_1,\ldots, x_{2n}\:|\:\b)
,\;(\text{even limit}).$$

We also define
$$\Gamma^{\L}_{}:=\lim_{\substack{\longleftarrow\\n}} 
\Gamma^{\L}(\x_n)\otimes_{\L} \L[[\b]]
\;\text { and }\;
\Gamma^{\L}_{+}:=\lim_{\substack{\longleftarrow\\n}} 
\Gamma^{\L}_{+}(\x_{n})\otimes_{\L} \L[[\b]].
$$

\begin{Def}(dual universal factorial Schur $P,Q$-functions)\\
We define $\widehat{p}^{\L}_\lambda(\y|\:\b)$ and $\widehat{q}^{\L}_\lambda(\y|\:\b)$
by the following identities. (Here $\mathcal{SP}$ is the set of all strict partitions.)
\begin{itemize}
\item[(1)]
$$\displaystyle
\Delta(\x;\y)=\prod_{i,j\geq 1}\frac{1-\overline{x}_i y_j}{1-x_i y_j}
=\sum_{\lambda\in \mathcal{SP}}
Q^\L_\lambda(\x\:|\:\b)\: \widehat{p}^\L_\lambda(\y|\:\b)
\hspace{1cm}
\text{(Cauchy identity)},
$$
\item[(2)]
$$\displaystyle
\Delta(\x;\y)=
\prod_{i,j\geq 1}\frac{1-\overline{x}_i y_j}{1-x_i y_j}
=\sum_{\lambda\in \mathcal{SP}}
P^\L_\lambda(\x|\:\b)\: \widehat{q}^\L_\lambda(\y|\:\b)
\hspace{1cm}
\text{(Cauchy identity)}.
$$
\end{itemize}
\end{Def}
\noindent

N.B. We can also define (type $A$) universal factorial Schur function 
$s^\L_\lambda(\x||\b)$ and its dual
$\widehat{s}^\L_\lambda(\y||\b)$  (see Appendix).

$\widehat{p}_{\lambda}^\L(\y|\b)$ and 
$\widehat{q}_{\lambda}^\L(\y|\b)$ are
formal series.
But
$\widehat{p}^\L_\lambda(\y)=\widehat{p}_{\lambda}^\L(\y|0)$ and 
$\widehat{q}^\L_\lambda(\y)=\widehat{q}_{\lambda}^\L(\y|0)$ are symmetric functions of finite degree
and 
top terms are usual Schur $P,Q$-functions $P_\lambda(\y)$ and $Q_\lambda(\y)$.

For example
$$\widehat{p}^\L_1(\y)=P_1(\y),\;\widehat{p}^\L_2(\y)=P_2(\y)+a_{1,1}h_1(\y),$$
$$\widehat{p}^\L_3(\y)=P_3(\y)+a_{1,1}h_2(\y)-2 a_{1,1} h_1^2(\y)+(a_{1,1}^2-a_{1,2})h_1(\y).$$

$$\widehat{q}^\L_1(\y)=Q_1(\y),\:
\widehat{q}^\L_2(\y)=Q_2(\y)-a_{1,1}h_1(\y),$$
$$\widehat{q}^\L_3(\y)=
Q_3(\y)+2 a_{1,1} h_2(\y)-3 a_{1,1} h_1(\y)^2+a_{1,1}^2 h_1(\y).$$

\begin{thm}(Basis theorem)
\begin{itemize}
\item[(0)] For a strict partition $\lambda$ of length $r$,
$\widehat{p}^\L_\lambda(\y|\b)\in \Gamma_{\L}^{(r)}(\y)$ and
$\widehat{q}^\L_\lambda(\y|\b)\in \Gamma_{\L}^{(r),+}(\y)$.

\item[(1)]
$\left\{\widehat{p}^\L_\lambda(\y|\b)\right\}_{\lambda\in \mathcal{SP}}$
are linearly independent and form an $\L[[\b]]$-basis
of $\Gamma_\L$.

\item[(2)]
$\left\{\widehat{q}^\L_\lambda(\y|\b)\right\}_{\lambda\in \mathcal{SP}}$ 
are linearly independent and form an $\L[[\b]]$-basis
of $\Gamma_\L^{+}$.

\end{itemize}
\end{thm}

\begin{proof}
(0) is a consequence of Prop. 3.1 (2).
(1) and (2) can be proved by the induction argument with regard to the
variables $\b$. %We omit the details.

\end{proof}

\subsection{Hopf algebra structure}

The ring of symmetric functions $\Lambda(\y)$ has a structure of Hopf algebra.
We can consider Hopf algebra structure on $\Gamma_\L$ by scalar extension and restriction.
Then we can consider $\Gamma^{\L}_{+}$ 
(resp. $\Gamma^{\L}_{}$)as a dual Hopf algebra
of $\Gamma_{\L}$
(resp. $\Gamma_{\L}^{+}$)
 over $\L[[\b]]$.
We will write $\phi$ for the coproduct maps.

\begin{prop}(duality)\\
\begin{itemize}
\item[(1)]

If \;$Q^\L_\lambda(\x|\b) Q^\L_\mu(\x|\b)=\displaystyle
\sum_{\nu\in \mathcal{SP}} c_{\lambda,\mu}^\nu(\b)\; Q^\L_\nu(\x|\b)$, then
$$\phi(\widehat{p}^\L_\nu(\y|\b))=
\sum_{\lambda,\mu\in \mathcal{SP}}c_{\lambda,\mu}^\nu(\b) \;
\widehat{p}^\L_\lambda(\y|\b)\otimes \widehat{p}^\L_\mu(\y|\b).$$

\item[(2)]
If \;$\widehat{p}^\L_\lambda(\y|\b) \widehat{p}^\L_\mu(\y|\b)=\displaystyle
\sum_{\nu\in \mathcal{SP}} 
\widehat{c}_{\lambda,\mu}^\nu(\b)\; \widehat{p}^\L_\nu(\y|\b)$, then
$$\phi(Q^\L_\nu(\x|\b))=
\sum_{\lambda,\mu\in \mathcal{SP}}
\widehat{c}_{\lambda,\mu}^\nu(\b) \;
Q^\L_\lambda(\x|\b)\otimes Q^\L_\mu(\x|\b).$$

\item[(3)]
If \;$P^\L_\lambda(\x|\b) P^\L_\mu(\x|\b)=\displaystyle
\sum_{\nu\in \mathcal{SP}} d_{\lambda,\mu}^\nu(\b)\; P^\L_\nu(\x|\b)$, then
$$\phi(\widehat{q}^\L_\nu(\y|\b))=
\sum_{\lambda,\mu\in \mathcal{SP}} d_{\lambda,\mu}^\nu(\b) \;
\widehat{q}^\L_\lambda(\y|\b)\otimes \widehat{q}^\L_\mu(\y|\b).$$

\item[(4)]

If \;$\widehat{q}^\L_\lambda(\y|\b) \widehat{q}^\L_\mu(\y|\b)=\displaystyle
\sum_{\nu\in \mathcal{SP}} 
\widehat{d}_{\lambda,\mu}^\nu(\b)\; \widehat{q}^\L_\nu(\y|\b)$, then
$$\phi(P^\L_\nu(\x|\b))=
\sum_{\lambda,\mu\in \mathcal{SP}}
\widehat{d}_{\lambda,\mu}^\nu(\b) \;
P^\L_\lambda(\x|\b)\otimes P^\L_\mu(\x|\b).$$

\end{itemize}
\end{prop}

\begin{proof}
These are formal consequences of the Cauchy identity (cf.\cite{Mol}).
\end{proof}

\section{Geometric background}

We will briefly explain the geometric meaning of the polynomials we defined.
The details will be explained in \cite{Nak-Nar}.

In \cite{H-H-H} the equivariant (generalized) cohomology ring 
(cf.\cite{May},\cite{Ada})
of a $T$-space $X$
is characterized as a GKM-space.
The vanishing property
(Prop.
\ref{vanishing}) means
that $P^\L_\lambda(\x|\:\b)$ 
has the property of free topological module generator $x_v$
(\cite{H-H-H} Prop. 4.1).
 
For the case of $K$-theory, the functions $P^\L_\lambda(\x|\b)$ and 
$Q^\L_\lambda(\x|\b)$ specialize to
$GP_\lambda(\x|\b)$ and $GQ_\lambda(\x|\b)$, $K$-theory factorial $P$- and $Q$-
functions (\cite{Ike-Nar} Def. 2.1).
These are shown to be the Schubert basis for the torus equivariant $K$-theory
of Lagrangian or orthogonal Grassmannians in \cite{Ike-Nar}. We will discuss
$K$-theory homology Schubert basis in  section 5.

\section{$K$-theory case}%sec4%%%%%%%%%%%%%%%%%%%%%%%%%%%%%%%%%%%%%%%%%%%%%%%%%%%

For the notations of root systems and Weyl groups of type $B,C,D$, we use the
convention of \cite{Ike-Nar}. 
Formal group law for $K$-theory is $F(u,v)=u+v+\beta u v$ ($\beta$ is 
an invertible element).
Using divided difference operators acting on the coefficient ring
$\Z[\b, \overline{\b}]$
 and the space of $K$-supersymmetric functions in $\x$-variables,
we can recursively compute the dual factorial $P,Q$-functions
 $\widehat{p}_\lambda^K(\y|\b)$ and $\widehat{q}_\lambda^K(\y|\b)$ for $K$-theory.
We will mainly explain this for the case of type $C$.
The root system has infinite simple roots $\{\alpha_i\}_{i=0,1,2,,\ldots}$ and
corresponding simple reflections $\{s_i\}_{i=0,1,2,,\ldots}$, which generate
the Weyl group $W(C_\infty)$.
The Coxeter relations are $(s_0 s_1)^4=(s_i s_{i+1})^3=(s_j s_k)^2=1$ for 
$1\leq i$ and $0\leq j<k-1$. 
The elements $e(\alpha_i)\in \Z[\beta][\b,\overline{\b}]$ corresponding to
each simple roots are defined as follows.
$$
e(\alpha_i)=b_{i+1}+_F \bar{b}_i \; (i \geq 1)
\text{ and } 
e(\alpha_0)=b_1+_F b_1  .$$

%$e(\alpha_{\hat{1}})=b_1+_F b_2$,

Let $\psi_i(f):=\frac{s_i(f)-f}{e(\alpha_i)}$. 
(This divided difference operator was used in \cite{Kir-Nar} to study Grothendieck
polynomial and its dual.)
Actually $\psi_i=\pi_i+\beta$
, where $\pi_i$ is the divided difference operator defined in \cite{Ike-Nar}.
We assume the following property of $K$-theoretic factorial
 Schur $Q$-function $Q^K_\lambda(\x|\b)=GQ_\lambda(\x|\b)$.
(see \cite{Ike-Nar} Theorem 6.1).
\begin{center}
$\psi_i(Q^K_\lambda(\x|\b))=0$ if $s_i\lambda\geq \lambda$ and
$\psi_i(Q^K_\lambda(\x|\b))=Q^K_{s_i\lambda}(\x|\b)+\beta Q^K_\lambda(\x|\b)$ 
if $s_i\lambda< \lambda$.
\end{center}
Here  Weyl group $W(C_\infty)$ naturally acts on the set of strict partitions $\mathcal{SP}$
(Actually $\mathcal{SP}$ can be considered as the set of  minimal coset 
representatives $W(C_\infty)/W(A_\infty)$, where $W(A_\infty)$ is the subgroup
genereted by $s_i (i\geq 1)$).
We write this action by $s_i\lambda$ for $\lambda\in \mathcal{SP}$.
We also define 
$\widehat{\psi}_i:=-s_i \psi_i$, because this will be suitable
for homology
as we see below:

\begin{prop} We put $\Delta=\Delta(\x;\y)$. We have the following formulas:
$$\widehat{\psi}_i\left(\frac{\widehat{p}^K_\lambda(\y|\b)}{\Delta}\right)
=
\left\{
\begin{array}{ccl}
\beta \:\frac{\widehat{p}^K_\lambda(\y|\b)}{\Delta}&&s_i\lambda<\lambda,\\[0.2cm]
0&&s_i\lambda=\lambda,\\[0.2cm]
\frac{\widehat{p}^K_{s_i\lambda}(\y|\b)}{\Delta}&&s_i\lambda>\lambda.\\
\end{array}
\right. 
$$
\end{prop}

\begin{proof}

By definition $\Delta=\displaystyle\prod_{i,j\geq 1}\frac{1-\bar{x}_i y_j}{1-x_i y_j}$
and
$\displaystyle
\Delta
=\sum_{\lambda\in \mathcal{SP}}
Q^K_\lambda(\x|\b) \widehat{p}^K_\lambda(\y|\b)
$
.
Therefore we have
$$\displaystyle
1
=\sum_{\lambda\in \mathcal{SP}}
Q^K_\lambda(\x|\b) \frac{\widehat{p}^K_\lambda(\y|\b)}{\Delta}
.$$

By the Leibniz rule  $\psi_i(f g)=\psi_i(f) s_i(g)+f \psi_i(g)$ and 
$\psi_i(h)=0$ if $s_i (h)=h	$,
we have

$$0=\sum_{\lambda} 
\psi_i(Q^K_\lambda(\x|\b)) s_i\left( \frac{\widehat{p}^K_\lambda(\y|\b)}{\Delta}\right)
+
\sum_{\lambda} 
Q^K_\lambda(\x|\b) \psi_i
\left(\frac{\widehat{p}^K_\lambda(\y|\b)}{\Delta}\right).
$$

Comparing the coefficients of $Q^K_\lambda(\x|\b)$,  we get

$$
\begin{array}{ccll}
0&=&s_i\left( \frac{\widehat{p}^K_\lambda(\y|\b)}{\Delta}\right)\beta
+\psi_i(\frac{\widehat{p}^K_\lambda(\y|\b)}{\Delta})
&\text{ if } s_i \lambda < \lambda,\\
0&=&\psi_i(\frac{\widehat{p}^K_\lambda(\y|\b)}{\Delta})
&\text{ if } s_i \lambda = \lambda,\\
0&=&s_i\left( 
\frac{\widehat{p}^K_{s_i\lambda}(\y|\b)}{\Delta}\right)
+\psi_i(\frac{\widehat{p}^K_\lambda(\y|\b)}{\Delta})
&\text{ if } s_i \lambda > \lambda.
\end{array}$$

\end{proof}

\vspace{0.3cm}

Iterating use of this proposition give the formula
$\widehat{\psi}_{w^C(\lambda)}(\frac{1}{\Delta})=
\widehat{p}^K_\lambda(\y|\b)(\frac{1}{\Delta})$,
where $w^C(\lambda)$ is the Weyl group element corresponding to 
$\lambda$ and $\widehat{\psi}_{w^C(\lambda)}$ is the product of operators
corresponding to a reduced expression of $w^C(\lambda)$.
Likewise we can use type $D_\infty$-Weyl group and get
$\widehat{\psi}_{w^D(\lambda)}(\frac{1}{\Delta})=\widehat{q}^K_\lambda(\y|\b)(\frac{1}{\Delta})$.
For type $A$ case,
$\widehat{s}^K_\lambda(\y||\b)$ can also be written by 
$\Delta^A=\displaystyle
\prod_{i,j}\frac{1-\bar{b}_i y_j}{1-x_i y_j}$, i.e.
$\widehat{\psi}_{w^A(\lambda)}(\frac{1}{\Delta^A})=
\widehat{s}^K_\lambda(\y||\b)(\frac{1}{\Delta^A})$.

{\bf Remark 1}:
We can define divided difference operator for the case of generalized cohomology.
(\cite{Bre-Eve,Bre-Eve2})
$$\widehat{\psi}_i (f):=\frac{s_i(f)-f}{e(-\alpha_i)}.$$
We can use this to calculate one row dual universal functions
$\widehat{p}^{\L}_{k}(\y|\b)$ and $\widehat{q}^{\L}_{k}(\y|\b)$ as the same rule.

{\bf Remark 2}:
For the case of usual homology,
the dual factorial Schur $P,Q$-functions are studied in  
\cite{Nar}. 

%\end{remark}

%\vspace{1cm}

\subsection{
{\bf Conjectural combinatorial formula} for 
\;$\widehat{p}^K_\lambda(\y)$ 
and 
\;$\widehat{q}^K_\lambda(\y)$.}

\begin{Def}
%{\bf Definition 1} 
(Tableaux)

For a strict partition $\lambda=(\lambda_1>\ldots>\lambda_r>0)$,
we define $Tab(\lambda)$ as the set of tableaux of shape
$\lambda$ in alphabet $1'<1<2'<2<\cdots$ with condition 
that each rows and columns are weakly increasing.
Let $Tab'(\lambda)$ be the subset of $Tab(\lambda)$
with the property that for each row the leftmost box contains a primed number.
\end{Def}

%\vspace{0.5cm}

%\noindent 
%{\bf Definition 2} $gp_\lambda, gq_\lambda$ 
Motivated by the construction of the dual stable Grothendieck polynomials
$g_\lambda(\y)$\;$(\lambda\in \mathcal{P})$\;
due to Lam-Pylyavskyy \cite{Lam-Pyl},
we define $gp_\lambda(\y), gq_\lambda(\y)$ as follows.

\begin{Def}
 For a strict partition $\lambda\in \mathcal{SP}$,
we define
$$gp_\lambda(\y):=\displaystyle\sum_{T\in Tab'(\lambda)} \y^T,\;
gq_\lambda(\y):=\displaystyle\sum_{T\in Tab(\lambda)} \y^T.$$

Here we define $\y^T=\displaystyle\prod_{i\in T} y_i^{T_C(i)} \prod_{i'\in T} y_i^{T_R(i')}$, where
$T_C(i)$ is the number of columns containing $i$ in $T$ and
$T_R(i')$ is the number of rows containing $i'$ in $T$.
\end{Def}
%\vspace{0.5cm}

%\noindent
%We also define $\tilde{gp_\lambda}(x):=(-1)^{|\lambda|}gp_\lambda(-x)$
%,$\tilde{gq_\lambda}(x):=(-1)^{|\lambda|}gq_\lambda(-x)$.

%\vspace{1cm}

\begin{conj}

$gp_\lambda(\y)=\widehat{p}^K_\lambda(\y)|_{\beta=-1}$
and
$gq_\lambda(\y)=\widehat{q}^K_\lambda(\y)|_{\beta=-1}$.

\end{conj}

%{\bf Remark}
%\begin{rem}
The conjecture above is shown to be true for one  row case.

\begin{prop}
$$gp_{k}(\y)=\widehat{p}^K_{k}(\y)|_{\beta=-1},
\text{ and }
gq_{k}(\y)=\widehat{q}^K_{k}(\y)|_{\beta=-1}
\text{ for k=1,2,3,\ldots}.
.$$
\end{prop}
%\begin{proof}

%\end{proof}
%\vspace{0.5cm}

There is a formula which 
relates $gp_{k}(\y)$ to 
the stable dual Grothendieck polynomials $g_\lambda(\y)$.

\begin{prop}(hook sum formula)
$$gp_k(\y)=\displaystyle\sum_{a=1}^k g_{a1^{k-a}}(\y).$$
\end{prop}
\begin{proof}
We can easily make a bijection between tableaux of their tableaux formulas.
\cite{Lam-Pyl}
\end{proof}
%\vspace{0.5cm}
\begin{conj} (for staircase)
\begin{center}
$gp_{\rho_k}(\y)=g_{\rho_k}(\y)$ for $\rho_k=(k,k-1,\ldots,2,1)$.
\end{center}
\end{conj}
\vspace{0.5cm}

{\bf Remark}:

For a geometric reason,
$gp_\lambda(\y)$ and $gq_\lambda(\y)$ should be 
(signed) positive linear combinations of dual
Grothendieck polynomials $g_\mu(\y)$ (hence Schur polynomials $s_\mu(\y)$).

\section{Appendix}

Here we use doubly infinite sequence $\b_{\pm}=(\ldots,b_{-2},b_{-1},b_0,b_{1},b_{2},\ldots)$.
Using Lazard ring $\L$ it is possible to extend Molev's dual Schur functions
$\widehat{s}_\lambda(\x||{\bf a})$
\cite{Mol}
to universal setting  as follows.
(A geometric meaning of Molev's dual Schur function is explained in \cite{Lam-Shi}.) 
%(e.g. root system of type $A_\infty$).
For a partition 
$\lambda=(\lambda_1\geq\lambda_2\geq\cdots\geq\lambda_n\geq 0)\in \mathcal{P}_n$,
we define
$$s^\L_\lambda(x_1,\ldots,x_n||\b_{\pm}):=\displaystyle
\sum_{w\in S_n} w\left(\frac{(x_1||\:\b_{\pm})^{\lambda_1+{n-1}}_{n}
(x_2||\:\b_{\pm})^{\lambda_2+{n-2}}_{n}
\cdots
(x_n||\:\b_{\pm})^{\lambda_{n}}_{n}
}{\displaystyle\prod_{1\leq i<j\leq n} (x_i+_F \overline{x}_j)} \right),$$

where
$(t\:||\:\b_{\pm})^{k}_{n}:=\displaystyle\prod_{i=1}^{k} (t\pf b_{n+1-i})$.

For a partition $\mu$, we define
a sequence as
$\overline{b}_{I-\mu}:=(\overline{b}_{1-\mu_1},\overline{b}_{2-\mu_2},\ldots)$.

\begin{prop}(vanishing)
\begin{itemize}
\item[(1)]
If $\mu\not\supset \lambda$ then
\hspace{2.5cm}
$
s^\L_\lambda(\overline{b}_{I-\mu}||\:\b_{\pm})=0.
$
\item[(2)]
$$
s^\L_\lambda(\overline{b}_{I-\lambda}||\:\b_{\pm})=
\prod_{(i,j)\in \lambda}(\overline{b}_{i-\lambda_{i}}\pf b_{\lambda'_{j}-j+1} ).
$$

\end{itemize}
\end{prop}

Then we define $\widehat{s}^\L_\lambda(\y||\b_{\pm})$ as

$$
\prod_{i=1}^{n}\prod_{j=1}^{\infty}\frac{1-\overline{b}_i y_j}{1-x_i y_j}
=
\sum_{\lambda\in P_n}
s^\L_\lambda(x_1,\ldots,x_n||\:\b_{\pm}) \:\widehat{s}^\L_\lambda(\y||\:\b_{\pm}).
$$

We can also take limit $n\to \infty$ to get $s^\L_\lambda(\x||\:\b_{\pm})$.
(specialization map is  the evaluation $x_n=\overline{b}_n$).

\vspace{1cm}

%%%%%%%%%%%%%%%%%%%%%%%%%%%%%%%%%%%%%%%%%%%%%%%%
%\pagebreak


\begin{thebibliography}{99}
 
\bibitem[Ada74]{Ada}
    J.\:F.\:Adams, Stable Homotopy and 
    Generalised homology theory,
    Chicago Lecture note.(1974)
    
\bibitem[Bak86]{Bak}
A.\:Baker, 
On the spaces classifying complex vector bundles with given real dimension, 
Rocky Mountain J. of Math.
16 (1986), 703--716.

\bibitem[Bre-Eve90]{Bre-Eve}
P.\:Bressler and  S. Evens,
The Schubert calculus, braid relations, and generalized cohomology, Trans. Amer. Math. Soc. 317 (1990), no. 2, 799--811.

\bibitem[Bre-Eve92]{Bre-Eve2}
P.\:Bressler and  S.\:Evens,
Schubert calculus in complex cobordism, Trans. Amer. Math. Soc. 331 (1992), no. 2, 799--813.

\bibitem[Cla81]{Cla}
 F.\:Clarke, 
The $K$-theory of $\Omega Sp(n)$, 
Quart. J. Math. Oxford Ser. (2) 32 (1981), 11--22.

\bibitem[HHH05]{H-H-H}
M.\:Harada, A. \:Henriques, and T.\:S.\:Holm,
Computation of generalized equivariant
cohomologies of Kac-Moody flag varieties,
Adv. Math. 197(2005), 198--221.


\bibitem[Ike-Nar]{Ike-Nar}
    T.\:Ikeda and H.\:Naruse,
    $K$-theoretic analogue of factorial Schur $P$- and $Q$-functions,
    arXiv:1112.5223 v2.
    
\bibitem[Kir-Nar]{Kir-Nar}
    A.\:N.\:Kirillov and H.\:Naruse,
    Construction of double Grothendieck polynomials
of classical type
using Id-Coxeter algebras.
    preprint.

\bibitem[Lam-Pyl07]{Lam-Pyl}
    T.\:Lam and P.\:Pylyavskyy,
    Combinatorial Hopf algebras and $K$-homology of Grassmannians,
    IMRN col.2007, rnm 125, arXiv:0705.2189.
 
    
\bibitem[Lam-Shi]{Lam-Shi}
    T.\:Lam and M.\:Shimozono,
    $k$-Double Schur functions and equivariant (co)homology of
the affine Grassmannian, 
    arXiv:1105.2170.

\bibitem[Laz55]{Laz}
M.\:Lazard,
Sur les groups de Lie formels \`a un param\`etre,
Bull. Soc. Math. France, 83 (1955), 251--274.

\bibitem[Len98]{Len}
C.\:Lenart,
Symmetric functions, formal group laws, and Lazards's theorem,
Adv. Math. 134 (1998), 219--239.

\bibitem[Lev-Mor07]{Lev-Mor}
M.\:Levine and F.\:Morel,
Algebraic Cobordism,
Springer Monographs in Math. 2007.

\bibitem[May96]{May}
J.\:P.\:May, Equivariant homotopy and cohomology theory, vol. 91, 
CBMS Regional Conference
Series, 1996.

   
\bibitem[Mol09]{Mol}
    A.\:Molev,
    Comultiplication rules for the double Schur functions and
    Cauchy identities,
    Elec. J. of Comb. 16 (2009), \#R13
    arXiv:0807.2127.
    
\bibitem[Nak]{Nak}
M.\:Nakagawa,
$K$-homology of $\Omega Sp$ and  $\Omega_{0} SO$,
note (Draft).

\bibitem[Nak-Nar]{Nak-Nar}
M.\:Nakagawa and H.\:Naruse,
Article in the Proceedings of MSJ-SI 2012 Schubert calculus, Osaka,
In preparation.



\bibitem[Nar]{Nar}
H.\:Naruse,
Homology factorial Schur $P,Q$-functions.
In preparation.

\bibitem[Pra91]{Pra}
P.\:Pagacz, Algebro-geometric applications of 
Schur $S$- and 
$Q$-polynomials. Topics in invariant theory (Paris, 1989/1990), 130--191, 
Lecture Notes in Math., 1478, Springer, Berlin, 1991.

\end{thebibliography}
\end{document}